\documentclass{article}
\title{Between the conjectures of P\'{o}lya and Tur\'{a}n}
\author{T. S. Trudgian\\
  Department of Mathematics and Computer Science,\\
  University of Lethbridge,\\
  Alberta, Canada, T1K 3M4\\
  \texttt{tim.trudgian@uleth.ca}}
\usepackage{amsthm}
\usepackage{amsmath}
\usepackage{amssymb}
\newtheorem*{conj}{Conjecture}
\newtheorem*{conje}{\textbf{The $\alpha = \frac{1}{2}$} Conjecture}
\newtheorem*{theorem}{Theorem}
\begin{document}
\maketitle

\begin{abstract}
This paper is concerned with the constancy in the sign of $L(X, \alpha) = \sum_{1}^{X} \frac{\lambda(n)}{n^{\alpha}}$, where $\lambda(n)$ the Liouville function. The non-positivity of $L(X, 0)$ is the P\'{o}lya conjecture, and the non-negativity of $L(X, 1)$ is the Tur\'{a}n conjecture --- both of which are false. By constructing an auxiliary function, evidence is provided that $L(X, \frac{1}{2})$ is the best contender for constancy in sign. The core of this paper is the conjecture that $L(X, \frac{1}{2}) \leq 0$ for all $X\geq 17$: this has been verified for $X\leq 300,001$.
\end{abstract}

\section{Introduction}
Let $\lambda(n)$ denote the Liouville function defined as $\lambda(n) = (-1)^{\Omega(n)},$ where $\Omega(n)$ counts, with multiplicity, the number of prime factors of $n$. It is well-known (see, e.g.\ \cite[Ch.\ I]{Titchmarsh}) that, for $\Re(s)>1$
\begin{equation}\label{open}
\sum_{n=1}^{\infty} \frac{\lambda(n)}{n^{s}} = \frac{\zeta(2s)}{\zeta(s)}.
\end{equation}
P\'{o}lya \cite{Polya} showed that, if $L(X) = \sum_{n\leq X} \lambda(n)$ is of a constant sign (for $X$ sufficiently large), then this implies the Riemann hypothesis. This lead to the 
\begin{conj}[The P\'{o}lya Conjecture]
\begin{equation*}
\sum_{n\leq X} \lambda(n) \leq 0,
\end{equation*}
for all $X\geq 2$.
\end{conj}
Similarly, Tur\'{a}n \cite{Turan} showed that the Riemann hypothesis is true if the sum\footnote{Actually, Tur\'{a}n, in his work with the partial sums of the zeta-function, $U_{n} (s) = \sum_{1}^{n}\nu^{-s}$, showed that the Riemann hypothesis is true if $U_{n}(s)$ has no zeroes in the region $\sigma>1$. This was shown to be false by Montgomery \cite{MontyU} in 1983. The non-vanishing of $U_{n}(s)$ in $\sigma >1$ is equivalent to $\sum_{1}^{X} \lambda(n)n^{-s} \geq 0$, for $X$ sufficiently large, for all $\sigma \geq 1$ --- see the remarks after (\ref{prod}).} $T(X) = \sum_{n\leq X} \frac{\lambda(n)}{n}$ is of constant sign (for $X$ sufficiently large). This lead to the 
\begin{conj}[The Tur\'{a}n Conjecture]
\begin{equation*}
\sum_{n\leq X} \frac{\lambda(n)}{n} \geq 0,
\end{equation*}
for all $X\geq 0$.
\end{conj}

In 1958 Haselgrove \cite{HaselP} showed that both of these conjectures are false, in spite of rather extensive numerical verification. The first value of $X$ (other than 1!\footnote{Fortunately, there is none of the customary ambiguity with the exclamation mark and factorial sign, nor for that matter, with the marking of this footnote.}) for which $L(X)> 0$ is $X= 906,105,257$, which was proved by Tanaka \cite{Tanaka}; the least $X$ for which $T(X) <0$ is greater than $7\cdot 10^{13}$, and was calculated explicitly by Borwein, Ferguson and Mossinghof \cite{Borwein}.

Here the following sum will be considered
\begin{equation*}
L(X, \alpha) = \sum_{n\leq X} \frac{\lambda(n)}{n^\alpha}, 
\end{equation*}
where $\alpha\geq 0$. The case $\alpha>1$ can be dispensed with easily. For a fixed $\alpha>1$ one can consider the Euler product of (\ref{open}), viz.
\begin{equation}\label{prod}
\sum_{n=1}^{\infty} \frac{\lambda(n)}{n^{\alpha}} = \prod_{p} \frac{\left(1-\frac{1}{p^{\alpha}}\right)}{\left(1-\frac{1}{p^{2a}}\right)} = \prod_{p} \left(1+\frac{1}{p^{\alpha}}\right)^{-1}.
\end{equation}
This is increasing with $\alpha$, whence one can use partial summation to show  there is a sufficiently large number, $X_{0}$ (which depends on $\alpha$) such that, for $X> X_{0}$ the sum $L(x, \alpha) \geq 0$.

Of interest, therefore, are in intemediate cases $0<\alpha < 1$. The approach taken here follows that taken by Haselgrove [\textit{op.\ cit.}] and is based on a theorem of Ingham \cite{Ingham1942}. Suppose one proposes the `$\alpha_+$ Conjecture', which says that, for $X$ sufficiently large, $L(X, \alpha) \geq 0$. To disprove such a conjecture one creates an auxiliary function which approximates $L(X, \alpha)$, whence, if this auxiliary function were negative in a certain region, then $L(X, \alpha)$ would likely be negative there as well. The locating of such a value (or the determining of an interval in which such a value lies) establishes the falsity of the $\alpha_+$ Conjecture. 

Despite the retrograde approach, it will be shown in the next section that the creating of this auxiliary function gives an insight into the expected behaviour of $L(X, \alpha)$. It is difficult to predict, \textit{ab initio}, those values of $\alpha$ for which there is a preponderance of negative (or positive) values of $L(X, \alpha)$; a na\:{i}ve computation of various values of $\alpha$ for, say, $X \leq 100$ is not very illuminating. 

\section{The work of Ingham}
Let $F(s, \alpha) = \int_{0}^{\infty} A(u, \alpha) e^{-su}\, du$, where $A(u, \alpha)$ is absolutely integrable over every finite interval $0\leq u \leq U$, and suppose the integral is convergent in the half-plane $\sigma >\frac{1}{2} - \alpha \geq 0$. Write $\rho_{n} = \frac{1}{2} + i\gamma_{n}$ for the $n$th zero of $\zeta(s)$, where $\gamma_{-n} = -\gamma_{n}$ and define $\gamma_{0} = 0$. Suppose $F(s, \alpha)$ has poles at $s= \rho_{n}-\alpha$ for $-N\leq n \leq N$, with residues labeled $r_{n} = r_{n}(\alpha)$. Define, for $\sigma > \frac{1}{2} - \alpha$, the function $F^{*}(s, \alpha) = \sum_{-N}^{N} \frac{r_{n}}{s-(\rho_{n} - \alpha)}$. Thus the function $F(s, \alpha) - F^{*}(s, \alpha)$ is regular in the region $\sigma > \frac{1}{2}-\alpha$. If one writes $A^{*}(u) = e^{(1/2 - \alpha)u} \sum_{-N}^{N} r_{n}e^{i\gamma_{n}u}$, then it is easily seen that $F^{*}(s) = \int_{0}^{\infty}A^{*}(u) e^{-su}\, du$. Then we have the following
\begin{theorem}[Ingham]\label{IngThm}
For a fixed $T$, as $u\rightarrow\infty$
\begin{equation}\label{bothwa}
\underline{\lim}\, e^{(\alpha-1/2)u} A(u, \alpha) \leq \underline{\overline{\lim}}\, A^{*}_T (u, \alpha) \leq \overline{\lim}\, e^{(\alpha-1/2)u} A(u,\alpha),
\end{equation}
where $A^{*}_T (u,\alpha)$ is a smoothed trigonometric polynomial defined as
\begin{equation*}
A^{*}_T (u, \alpha) =r_{0} + 2\Re \sum_{0<\gamma_{n} < T} \left( 1-\frac{\gamma_{n}}{T}\right)r_{n}e^{i\gamma_{n}u}.
\end{equation*}
\end{theorem}
\begin{proof}
See Ingham [\textit{op.\ cit.}] 
\end{proof}
For the present purpose take $A(u) = L(e^{u}, \alpha)$ --- see (\ref{Fdef}). If one were to find a value of $u$ and $T$ for which $A^{*}_T (u) >0$ then $L(e^{u}, \alpha) >0$, with a similar property holding in the other direction, courtesy of (\ref{bothwa}). Henceforth $A^{*}_T (u)$, which will of course depend of $\alpha$, will be called the `auxiliary polynomial'.

\section{Application of Ingham's Theorem}

For $\sigma > 1-\alpha$, one can use partial summation on (\ref{open}) to write
\begin{equation}\label{parts}
\frac{\zeta(2(\alpha +s))}{\zeta(\alpha +s)} = 
 \lim_{X\rightarrow\infty}\frac{\left(\sum_{n\leq X}\frac{\lambda(n)}{n^{\alpha}}\right)}{X^{s}} + s\int_{1}^{X} \left(\sum_{n\leq t}\frac{\lambda(n)}{n^{\alpha}}\right)t^{-(s+1)}\, dt.
\end{equation}
It was first shown by Landau that, on the Riemann hypothesis, $\sum_{n\leq X} \lambda(n) = O(X^{1/2 + \epsilon})$. Thus the integral on the right side of (\ref{parts}) converges and defines an analytic function for $\sigma > \frac{1}{2} - \alpha$, whence it follows by analytic continuation that
\begin{equation}\label{Fdef}
F(s, \alpha) := \frac{\zeta(2(\alpha+s))}{s\zeta(\alpha+s)} = \int_{1}^{\infty} \left(\sum_{n\leq t}\frac{\lambda(n)}{n^{\alpha}}\right)t^{-(s+1)}\, dt.
\end{equation}
Possible poles of $F(s, \alpha)$ are at $s=0$, at $s= \rho_{n} - \alpha$ and at $s=\frac{1}{2} - \alpha$. When $\alpha = 1$ there is no pole at $s=0$, and when $\alpha<\frac{1}{2}$ the pole at $s=0$ lies outside the half-plane $\sigma > \frac{1}{2} - \alpha$. Thus, when $\frac{1}{2}<\alpha<1$,
\begin{equation*}
\textrm{Res}\left(F(s, \alpha); s= 0\right) = \lim_{s\rightarrow 0} \frac{\zeta(2s)}{\zeta(s)} = 1.
\end{equation*}
When $\alpha = \frac{1}{2}$, writing $\zeta(s) = (s-1)^{-1} + \gamma + O(s-1)$, where $\gamma$ is Euler's constant, gives
\begin{equation*}
\textrm{Res}\left(F(s, \alpha); s= 0\right) = \frac{\gamma}{\zeta(\frac{1}{2})}.
\end{equation*}
Lastly, assuming the simplicity of the zeroes,
\begin{equation*}
\textrm{Res}\left(F(s, \alpha); s= \rho_{n} - \alpha, n\neq 0\right) = \frac{\zeta(2\rho_{n})}{(\rho_{n} - \alpha)\zeta'(\rho_{n})}.
\end{equation*}
Thus the auxiliary polynomial becomes
\begin{equation}\label{auxe}
A_{T}^{*}(u, \alpha) = r_{0}(\alpha) + 2 \Re \sum_{0<\gamma_{n}<T}\frac{\zeta(2\rho_{n})}{(\rho_{n}-\alpha)\zeta'(\rho_{n})}\left\{1-\frac{\gamma_{n}}{T}\right\} e^{i\gamma_{n}u},
\end{equation}
where,
\begin{equation}\label{aux}
r_{0}(\alpha) = \begin{cases} 
\{(1-2\alpha)\zeta(\frac{1}{2})\}^{-1}, & \mbox{if } 0\leq \alpha < \frac{1}{2}, \, \alpha = 1\\
\{(1-2\alpha)\zeta(\frac{1}{2})\}^{-1} + 1, & \mbox{if } \frac{1}{2}<\alpha<1 \\
\gamma/\zeta(\frac{1}{2}), &  \mbox{if } \alpha = \frac{1}{2}.
\end{cases}
\end{equation}

Thus the $\alpha = 0, 1$ cases of (\ref{aux}) give, respectively, with the P\'{o}lya and the Tur\'{a}n auxiliary polynomials. Since $\zeta(\frac{1}{2}) <0$, the auxiliary polynomial $A_{T}^{*}(u, 0)$ consists of a negative term followed by a sum of oscillating terms. One might suspect therefore that the P\'{o}lya conjecture `should' be true, because the terms in the sum should not reinforce one another sufficiently often. 

The leading term in $A_{T}^{*}(u, \alpha)$ decreases without limit as $\alpha\uparrow \frac{1}{2}$. Thus, one might suspect that, say, $L(X, 0.499)$ might be positive for all sufficiently large $X$. Computations of this sort do not yield the `expected' result, perhaps due to the presence of the $p_{n}  - \alpha$ term in the denominator of (\ref{aux}).

Note though, when $\alpha = \frac{1}{2}$, the summands in (\ref{auxe}) are of the form
\begin{equation*}
\Re \frac{\zeta(1 + 2i\gamma_{n})}{i\gamma_{n}\zeta'(\frac{1}{2} + \gamma_{n})}\theta_{n},
\end{equation*}
where $|\theta_{n}| \leq 1$. Fujii \cite{FujiiShanks} has proved that $\zeta'(\frac{1}{2} +i\gamma_{n})$ is real positive in the mean. On the Riemann hypothesis, Littlewood (see, e.g.\ \cite[Ch.\ XIV]{Titchmarsh}) showed that $\zeta(1+ t) = O(\log\log t).$ Thus, not only should the summands in (\ref{auxe}) be small, but they are at their smallest when $\alpha = \frac{1}{2}.$ This suggests that perhaps the strongest case for a constancy in sign of the sum $L(X, \alpha)$ is made when $\alpha = \frac{1}{2}$. Using Mathematica, I have checked that the sum $L(X, \frac{1}{2})$ is negative for all $17 \leq X \leq 300,001$. Of interest would be the investigation of
\begin{conje}
\begin{equation*}
L(X, 1/2) = \sum_{1 \leq n \leq X} \frac{\lambda(n)}{n^{1/2}} \leq 0,
\end{equation*}
for all $X\geq 17$.
\end{conje}

\bibliographystyle{plain}
\bibliography{themastercanada}

\end{document}